\documentclass[12pt]{article}
\usepackage{ct_modified}

\usepackage{enumerate}
\usepackage{mathrsfs}
\usepackage{thmtools}
\usepackage{thm-restate}

\dateline{}{}{}


\title{A degree bound for planar functions}

\author[1]{Christof Beierle}
\author[2,1]{Tim Beyne\thanks{Supported by a junior postdoctoral fellowship from the Research Foundation -- Flanders (FWO) with reference number 1274724N.}}

\affil[1]{%
Faculty of Computer Science, Ruhr University Bochum, Bochum, Germany

\email{christof.beierle@rub.de}%
}

\affil[2]{%
COSIC, KU Leuven, Leuven, Belgium

\email{tim.beyne@esat.kuleuven.be}%
}

\newcommand{\comp}[1]{\bar{#1}}
\newcommand{\F}{\mathbb{F}}
\newcommand{\QQ}{\mathbb{Q}}
\newcommand{\ZZ}{\mathbb{Z}}
\newcommand{\CC}{\mathbb{C}}
\newcommand{\K}{\mathbb{K}}
\newcommand{\tr}{\mathrm{tr}}
\newcommand{\wt}{s_p}

\newcommand{\dual}{\vee}

\DeclareMathOperator{\ord}{ord}
\DeclareMathOperator{\FT}{\mathscr{F}}
\DeclareMathOperator{\UT}{\mathscr{U}}
\DeclareMathOperator{\TT}{\mathscr{T}}
\newcommand{\adeg}[1]{d_{\mathsf{alg}}(#1)}

\begin{document}

\maketitle

\begin{abstract}
  Using Stickelberger's theorem on Gauss sums, we show that if $F$ is a planar function on a finite field $\mathbb{F}_q$, then for all non-zero functions $G : \mathbb{F}_q \to \mathbb{F}_q$, we have
\begin{equation*}
d_{\mathsf{alg}}(G \circ F) - d_{\mathsf{alg}}(G) \le \frac{n(p-1)}{2},
\end{equation*}
where $q = p^n$ with $p$ a prime and $n$ a positive integer, and $d_{\mathsf{alg}}(F)$ is the algebraic degree of $F$, i.e., the maximum degree of the corresponding system of $n$ lowest-degree interpolating polynomials for $F$ considered as a function on $\mathbb{F}_p^n$. This bound implies the (known) classification of planar polynomials over $\mathbb{F}_p$ and planar monomials over $\mathbb{F}_{p^2}$. As a new result, using the same degree bound, we complete the classification of planar monomials for all $n = \smash{2^k}$ with $p>5$ and $k$ a non-negative integer. Finally, we state a conjecture on the sum of the base-$p$ digits of integers modulo $q-1$ that implies the complete classification of planar monomials over finite fields of characteristic $p>5$.
\end{abstract}

\section{Introduction}
Throughout this work, let $p$ be a prime and $n$ a positive integer. Let $\F_{q}$ denote a field with $q = p^n$ elements. A function $F\colon \F_{q} \rightarrow \F_{q}$ is called \emph{planar} if the mappings
\begin{align}\label{eq:derivatives}x \mapsto F(x+\alpha) - F(x) -F(\alpha)\end{align}
are permutations on $\F_{q}$ for all non-zero $\alpha$ in $\F_{q}$. The notion of planar functions originally comes from finite geometry and goes back to Dembowski and Ostrom~\cite{dembowski1968planes}, who studied projective planes of finite order $n$ possessing a collineation group of order $n^2$. Planar functions have a strong relation to commutative semifields~\cite{coulter2008commutative} and (partial) difference sets~\cite{DBLP:journals/dcc/WengQWX07}, and have applications in coding theory~\cite{DBLP:journals/tit/CarletDY05}. 
They can only exist for odd characteristic $p$, as for $p=2$ any element in the image of the function defined in~\eqref{eq:derivatives} has at least two preimages $x$ and $x+\alpha$. If $p$ is odd, then there always exists a planar function, the canonical example being $x \mapsto x^2$. A complete classification is only known over prime fields, and was established independently by Gluck~\cite{DBLP:journals/dm/Gluck90}, Hiramine~\cite{hiramine1989conjecture}, and Ronyai and Sz{\H{o}}nyi~\cite{ronyai1989planar}. More precisely, a function over $\F_p$ with $p$ an odd prime is planar if and only if it is of the form $x \mapsto ax^2+bx+c$  with $a \neq 0$. For more details on planar functions, their properties and known families, we refer to the survey by Pott~\cite{DBLP:journals/dcc/Pott16}. 

There is a one-to-one correspondence between functions $F \colon \F_{q} \rightarrow \F_{q}$ and polynomials in the quotient ring $\F_{q}[X] /(X^{q}-X)$, as any function over $\F_{q}$ is the evaluation map of a unique interpolating polynomial of the form
\begin{align}
\label{eq:univariate_representation}
    \sum_{i=0}^{q-1} a_i X^i, \quad a_i \in \F_{q}.
\end{align}
A polynomial is called \emph{planar} if its evaluation map is planar. For odd $p$, a polynomial of the form
\[\sum_{0 \leq i \leq j \leq n-1} a_{i,j}X^{p^{i}+p^j}, \quad a_{i,j} \in \F_{q}\]
is called a \emph{Dembowski-Ostrom polynomial}. Planar Dembowski-Ostrom polynomials are an interesting special case of planar polynomials, as they are in one-to-one correspondence with commutative presemifields of odd order, by defining the presemifield operation $\odot$ corresponding to a planar Dembowski-Ostrom polynomial $F$ as $x \odot \alpha = F(x+\alpha) - F(x) - F(\alpha)$, see~\cite{coulter2008commutative}. In this work, Coulter and Henderson completely classified planar Dembowski-Ostrom polynomials over $\F_{p^2}$ and $\F_{p^3}$. 

In~\cite{dembowski1968planes}, Dembowski and Ostrom mentioned the possibility that every planar function can be represented as a Dembowski-Ostrom polynomial.\footnote{Adding a polynomial of the form $c + \sum_{i=0}^{n-1}a_iX^{p^i}$ with $c$ and $a_i$ in $\F_{q}$ does not affect the planar property, so this statement only makes sense for polynomials where the coefficients of $\smash{X^{p^i}}$ are zero and without constant term.} This conjecture was actually proven false by Coulter and Matthews in~\cite{DBLP:journals/dcc/CoulterM97} by showing that the monomial $X^{(3^i+1)/2}$ is planar over $\F_{3^n}$ if $\gcd(i,n) = 1$ and $i$ is odd. A special case of this family of counterexamples was independently discovered in~\cite{DBLP:journals/aaecc/HellesethS97}. However, the conjecture remains open for $p>3$ and (up to the notion of graph equivalence, i.e., two functions $F$ and $G$ on $\F_q$ are \emph{graph equivalent} if there exists an affine bijection on the $\F_p$-vector space $\F_q^2$ that maps the graph of $F$ to the graph of $G$) the counterexamples found in~\cite{DBLP:journals/dcc/CoulterM97} are the only known counterexamples to the Dembowski-Ostrom conjecture. 

There is an ongoing line of research aiming at the classification of planar monomials, i.e., planar polynomials of the form $X^d$. So far, the complete classification of planar monomials over $\F_{p^n}$ is known only for $n$ in $\{1,2,3,4\}$, see~\cite{johnson1987projective,coulter2006classification,DBLP:journals/ffa/BergmanCV22,DBLP:journals/ffa/CoulterL12}. Besides that, planar Dembowski-Ostrom monomials have been fully classified: a monomial $X^{p^i + p^j}$ over $\F_{q}$ with $0 \leq i \leq j \leq n-1$ is planar if and only if $n/\gcd(j-i,n)$ is odd~\cite{DBLP:journals/dcc/CoulterM97}. Hence, the Dembowski-Ostrom conjecture for monomials over fields of characteristic $p > 3$ can be stated as follows.

\begin{conjecture}
    \label{con:do} If $\F_q$ has characteristic $p>3$, then $X^d$ is a planar monomial over $\F_q$ if and only if $d \equiv p^i + p^j \pmod {q-1}$ with $n/\gcd(j-i,n)$ odd.
\end{conjecture}

Apart from the classification results for $n\leq 4$, the closest we came to resolving Conjecture~\ref{con:do} was a result by Zieve~\cite{zieve2015planar}, who proved that for a fixed prime $p$, there does not exists an exponent $d$ other than those mentioned in Conjecture~\ref{con:do} (and those of the form $(3^i+3^j)/2$ in the case of $p=3$) such that $X^d$ is planar over infinitely many fields $\F_{p^n}$.   Mentioning the result of Zieve, it is worth remarking that a result of Menichetti on division algebras implies that for $n$ a prime, all planar Dembowski-Ostrom polynomials with $p$ large enough correspond (up to a notion of equivalence) to planar Dembowski-Ostrom monomials~\cite{menichetti1996n}. Note that this is in contrast to the case of composite $n$, as it was shown in~\cite{golouglu2023exponential} that the number of (non-equivalent) planar Dembowski-Ostrom polynomials over $\F_{p^{4m}}$ grows exponentially in $m$.

\paragraph{Our Results.}
The aim of this work is to make progress on the classification of planar functions. We achieve this by proving that the algebraic degree of a planar function $F$ composed with an arbitrary non-zero function $G$ only grows additively with the algebraic degree of $G$.

For a non-negative integer $e$, let $\wt(e)$ be the sum of the digits in the base-$p$ representation of $e$. The \emph{algebraic degree} of a non-zero function $F \colon \F_{q} \rightarrow \F_{q}$, denoted $\adeg{F}$, is defined as the largest value $s_p(i)$ for which $X^i$ has a non-zero coefficient in the unique interpolating polynomial of degree at most $q-1$ for $F$. That is, in terms of representation~\eqref{eq:univariate_representation}, $\adeg{F} = \max\{\wt(i) \mid a_i \neq 0\}$. 
We prove the following result.
\begin{restatable}{theorem}{degreeBoundPlanar} \label{thm:degreeBoundPlanar}
If $F$ is a planar function on $\F_q$, then for all non-zero functions $G : \F_q \to \F_q$, we have
\begin{equation*}
\adeg{G \circ F} - \adeg{G} \le \frac{n(p - 1)}{2}.
\end{equation*}
\end{restatable}

The bound for the case that $G$ is the identity function is relatively high and was already known before in the context of bent functions. Recall that a function $f \colon \F_q \rightarrow \F_p$ is called \emph{bent} if
$\smash{\bigl\lvert \sum_{x \in \F_q} \theta^{\tr(f(x)-ux)} \bigr\rvert} = \sqrt{q}$ for all $u$ in $\F_q$, where $\theta$ is a primitive $p$-th root of unity in the complex numbers and $\tr$ denotes the absolute trace on $\F_{q}$.
Hou~\cite{DBLP:journals/ffa/Hou04} proved that a bent function $f \colon \F_{q} \rightarrow \F_p$ fulfills $\adeg{f} \leq n(p-1)/2+1$ (a function $F$ on $\F_{q}$ is planar if and only if all functions $x \mapsto \tr(v F(x))$ with $v$ in $\F_q^\times$ are bent, see~\cite{carlet2001generalized,DBLP:journals/dcc/Pott16}).  An analog of Theorem~\ref{thm:degreeBoundPlanar} for the case of $p=2$ in terms of divisibility of Walsh coefficients of $F$ was shown in~\cite{DBLP:conf/eurocrypt/CanteautV02}. The remarkable part of \Cref{thm:degreeBoundPlanar} is the fact that the algebraic degree of $G \circ F$ grows only additively with the algebraic degree of $G$. This is in contrast with what one expects for functions of low-algebraic degree, namely that the
algebraic degree should be around $\adeg{G} \cdot \adeg{F}$ for large enough $p$ and $F$ of constant algebraic degree. Note that the bound $\adeg{G \circ F} \leq \adeg{G} \cdot \adeg{F}$ implies that the inequality in Theorem~\ref{thm:degreeBoundPlanar} is  fulfilled for any function $F$ with $\adeg{F} \leq 2$, and can therefore only be used to rule out planarity if $\adeg{F} >2$. 

Let us define a binary operation $\star$ on the set $S = \{0,1,\dots,q-1\}$ by
\begin{align*}
    e \star d = \begin{cases}
        0   &\text{ if } 0 \in \{e,d\}\,, \\
        q-1 &\text{ if } 0 \notin \{e,d\} \text{ and } r = 0\,, \\
        r   &\text{ if } 0 \notin \{e,d\} \text{ and } r \neq 0,
    \end{cases}
\end{align*}
where $r$ is the unique non-negative integer smaller than $q-1$ such that $r \equiv ed \pmod {q-1}$. This operation makes $(S,\star)$ into a commutative monoid and we have $\adeg{X^{ed}} = \wt(e \star d)$. \Cref{thm:degreeBoundPlanar} then has the following corollary for planar monomials.

\begin{restatable}{cor}{degreeBoundPlanarMonomial} \label{cor:degreeBoundPlanarMonomial}
If $X^d$ is a planar monomial over $\F_q$ with $0 \le d \le q - 1$, then
\begin{equation*}
\wt(e \star d) - \wt(e) \le \frac{n(p - 1)}{2},
\end{equation*}
for all $e$ in $\{0, 1, \ldots, q - 1\}$.
\end{restatable}
This provides a method for proving the non-planarity of a monomial $X^d$, namely by finding an element $e$ with $1 \leq e \leq q-1$ that violates the degree bound, i.e., $\wt(e \star d) - \wt(e) > n(p-1)/2$. Using this technique, we establish again the classification of planar polynomials over $\F_{p}$ and planar monomials over $\F_{p^2}$. In addition, for the first time, we completely classify planar monomials over $\F_{p^{2^k}}$ for $p>5$.

\begin{restatable}{theorem}{powerTwo} \label{thm:main_power2}
Let $k$ be a non-negative integer and $p>5$. The monomial $X^d$ is planar over $\smash{\F_{p^{2^k}}}$ if and only if $d \equiv 2p^i \pmod{p^{2^k} - 1}$ for some non-negative integer $i$.
\end{restatable}

\begin{remark}
A bound similar to the one stated in \Cref{cor:degreeBoundPlanarMonomial} is known for $p=2$ in the context of the proof of the Niho conjecture on maximally nonlinear (almost perfect nonlinear, see e.g.~\cite{hou2004note} for a definition) monomials. More precisely, for $p=2$ and $n = 2m+1$, it is known that an almost perfect nonlinear monomial $X^d$ over $\F_{q}$ is maximally nonlinear if and only if $s_2(e \star d) - s_2(e) \leq m$, see~\cite{hollmann2001proof,hou2004note}. In those works, the proof of the maximal nonlinearity of a certain monomial was established by bounding above $s_2(e \star d) - s_2(e)$ for all $e$ with $1 \leq e \leq 2^n-1$.
\end{remark}

\begin{remark}
    A referee pointed us to~\cite{DBLP:journals/dcc/LangevinV05} for the first paper that uses Stickelberger's theorem directly to study cryptographic properties of monomials in characteristic two. Previous results of this kind resorted to McEliece's weight divisibility theorems for $p$-ary cyclic codes~\cite{MCELIECE197180,DBLP:journals/dm/McEliece72}. The result in~\cite{DBLP:journals/dm/McEliece72} states a general congruence relation for weights in a $p$-ary cyclic code, which can be seen as a generalization of Ax's theorem~\cite{ax1964zeroes}. Note that both~\cite{ax1964zeroes} and~\cite{DBLP:journals/dm/McEliece72} use Stickelberger's result.
\end{remark}

Finally, we state and discuss a conjecture on the sum of base $p$-digits of integers modulo $q-1$ (\Cref{conj:extended}). A proof of this conjecture implies the complete classification of planar monomials over finite fields of characteristic $p>5$.

\section{Degree Bound}
Our proof relies on some ideas that were developed in the context of symmetric-key cryptanalysis, where a variant of Lemma~\ref{lem:baseConversionBound} was proven for functions from $\F_2^n$ to $\F_2^m$ as an application of ultrametric integral cryptanalysis~\cite{Beyne2023,DBLP:conf/asiacrypt/BeyneV24}.
The basic idea is to associate a linear operator to a function on $\F_q$, and to compare the additive and multiplicative Fourier transformations of this operator. Importantly, both Fourier transformations will be defined over a suitable extension of the $p$-adic numbers. 

The motivation for this approach is that the property of being a planar function is inherently related to the additive structure of $\F_q$, and in particular has a simple characterization in terms of the additive Fourier transformation. The notion of algebraic degree is closely related to the multiplicative structure of $\F_q$, because the multiplicative characters of $\F_q$ are (lifted) monomial functions.

\subsection{Additive and Multiplicative Fourier Transformations}
We will define additive and multiplicative Fourier transformations over a local field. This requires some background from number theory, see for example~\cite[Chapter 1]{lang2012cyclotomic} and~\cite[Chapter 3]{koblitz2012p}.
Let $p$ be an odd prime, $\QQ_p$ the field of $p$-adic numbers, and $\F_q$ a field of order $q = p^n$. Let $\zeta_{q - 1}$ be a primitive $(q-1)$\textsuperscript{st} root of unity so that the algebraic extension $\QQ_p(\zeta_{q - 1})$ of $\QQ_p$ has residue field $\F_q$. Throughout this section, we work over the totally ramified extension $\K = \QQ_p(\zeta_{q - 1}, \zeta_p)$ of $\QQ_p(\zeta_{q - 1})$, with $\zeta_p$ a primitive $p$\textsuperscript{th} root of unity.
The field $\K$ is local with uniformizer $\pi$ equal to $\zeta_p - 1$. That is, every nonzero element $x$ in $\K$ can be written as $x = u\,\pi^i$ with $i$ a unique integer and $u$ a unit in the ring of integers of $\K$. 
The valuation of $x$ will be denoted as $\ord_{\pi} x = i$. Since $(\pi)^{p - 1} = (p)$ as ideals of the ring of integers of $\K$, the $p$-adic valuation extends to $\K$ by
\begin{equation*}
\ord_p x = \frac{\ord_{\pi} x}{\ord_{\pi} p} = \frac{\ord_{\pi} x}{p - 1}.
\end{equation*}
The corresponding $p$-adic absolute value of an element $x$ of $\K$ will be denoted by $|x|_p = p^{-\ord_p x}$.
By convention, $\ord_p 0 = \infty$ so that $|x|_p = 0$ if and only if $x = 0$. A similar setup was used in~\cite{DBLP:journals/ffa/Hou04} for the proof of the degree bound for bent functions, but can be traced back further and is already found in the work of Ax~\cite{ax1964zeroes}.

Let $\K[X]$ be the free $\K$-vector space on a finite commutative monoid $X$, and write $\K^X$ for the vector space of functions from $X$ to $\K$. By extending functions on $X$ linearly to all of $\K[X]$, we can think of $\K^X$ as the dual vector space of $\K[X]$.
To avoid confusion between $X$ and $\K[X]$, the standard basis vectors of $\K[X]$ will be denoted by $\delta_x$, where $x \in X$. The corresponding dual basis of $\K^X$ consists of the functions $\delta^x : X \to \K$ such that $\delta^x(y) = 1$ if $x = y$ and zero otherwise.

A character is a homomorphism of monoids $X \to \K$. By a well-known result of Dedekind \cite[\S44]{Dedekind}, characters are linearly independent. Furthermore, the characters of $X$ form a monoid $\widehat{X}$ under pointwise multiplication. It follows from the representation theory of monoids that if $\K$ contains enough roots of unity, then there exist precisely $|X|$ characters if and only if $X$ is a commutative \emph{inverse monoid}~\cite[\S5.2]{Steinberg2016}. A monoid $X$ is inverse if for every $x$ in $X$, there exists a $y$ in $X$ such that $xyx = x$. Hence, if $X$ is an inverse monoid and $\K$ contains enough roots of unity, then its characters form a basis for $\K^X$. This will be true in our setting.

Dually, for a character $\chi$, we can define $\chi^\dual$ as the unique element of $\K[X]$ such that $\psi(\chi^\dual) = 1$ if $\psi = \chi$ and 0 otherwise. Here, we consider $\psi$ as an element of the dual space of $\K[X]$.
We define the Fourier transformation as the change-of-basis transformation from the standard basis of $\K[X]$ to the basis $\smash{\big\{\chi^\dual~|~\chi \in \widehat{X}\big\}}$ as follows.

\begin{definition}[Fourier transformation] \label{def:fourierTransformation}
If $X$ is a finite commutative inverse monoid, then the Fourier transformation on $\K[X]$ is the linear map $\FT_X : \K[X] \to \K[\widehat{X}]$ defined by $\chi^\dual \mapsto \delta_\chi$ for all monoid characters $\chi$ in $\widehat{X}$.
\end{definition}

A common alternative to \Cref{def:fourierTransformation} is to define the Fourier transformation as the dual change-of-basis $\FT_X^{-\dual} : \smash{\K^X \to \K^{\widehat{X}}}$ that maps $\chi$ to $\delta^\chi$, where $\FT_X^{-\dual}$ is the inverse of the adjoint of $\FT_X$.
If $X$ is a group, then it is common practice to identify $\K[X]$ with its dual $\K^X$. Indeed, $\chi$ and its dual $\chi^\dual$ are then the same up to conjugation and scaling.
However, in the multiplicative case, the distinction will be significant for our purposes.

Throughout this paper, $X$ will either be the additive group or the multiplicative monoid of the field $\F_q$. The characters of the additive group are given by $\chi : x \mapsto \smash{\zeta_p^{\tr(ux)}}$ for $u$ in $\F_q$, and the corresponding dual basis element is
\begin{equation*}
\chi^\dual = \frac{1}{q}\sum_{x \in \F_q} \zeta_p^{-\tr(ux)}\,\delta_x.
\end{equation*}
The additive Fourier transformation will be denoted by $\FT$, with the field $\F_q$ assumed to be clear from context.
The multiplicative monoid $\F_q$ is inverse, and its characters are given by $\lambda : x \mapsto \tau(x^i)$ for $i$ in $\{0, 1, \ldots, q - 1\}$. Here, $\tau : \F_q \to \QQ_p(\zeta_{q - 1}) \subset \K$ is the Teichm\"uller character. Recall that $\F_q$ is the residue field of $\K$. By definition $\tau(0) = 0$ and, for $x \neq 0$, $\tau(x)$ is the unique $(q - 1)$\textsuperscript{st} root of unity such that $\tau(x) \equiv x \pmod{p}$. The existence and uniqueness are ensured by Hensel lifting. Restricting the characters of the monoid $\F_q$ to $\F_q^\times$ yields the characters of the multiplicative group $\F_q^\times$.
If $i \neq 0$, then the corresponding dual basis element is
\begin{equation*}
\lambda^\dual = \frac{1}{q - 1} \sum_{x \in \F_q^\times} \tau\big(x^{-i}\big)\,\delta_x + \begin{cases}
-\delta_0  & \text{if}~i = q - 1, \\
0          & \text{else}. \\
\end{cases}
\end{equation*}
If $i = 0$, i.e. $\lambda$ is the trivial character $x \mapsto 1$, then $\lambda^\dual = \delta_0$. The multiplicative Fourier transformation will be denoted by $\UT$, with the field $\F_q$ again assumed to be clear from the context.

\subsection{Linear Maps Corresponding to a Function on \texorpdfstring{$\F_q$}{a Finite Field}}
For a function $F : \F_q \to \F_q$, we define a linear map $T^F : \K[\F_q] \to \K[\F_q]$ by
\begin{equation*}
T^F \,\delta_x = \delta_{F(x)}.
\end{equation*}
The adjoint of $T^F$ is the linear map $T^{F^\dual} : \K^{\F_q} \to \K^{\F_q}$ such that $T^{F^\dual} f = f \circ F$.
The proof of our main result is based on comparing the additive and multiplicative Fourier transformation of the linear map $T^F$.
For brevity, in the following we identify linear operators with their matrix representation relative to the standard basis.
In the cryptanalysis literature, the additive Fourier transformation of $T^F$ is called the \emph{correlation matrix} $C^F$ of $F$, see~\cite{DBLP:conf/fse/DaemenGV94,DBLP:conf/asiacrypt/Beyne21}:
\begin{equation*}
C^F = \FT \,T^F \,\FT^{-1}.
\end{equation*}
It plays a central role in linear cryptanalysis, and its matrix coordinates at additive characters $\chi : x \mapsto \zeta_p^{\tr(ux)}$ and $\psi : x \mapsto \zeta_p^{\tr(vx)}$ are equal to
\begin{equation*}
C^F_{\psi, \chi} = \psi\big(T^F \chi^\dual\big) = \frac{1}{q} \sum_{x \in \F_q} \zeta_p^{\tr(v F(x) - ux)}.
\end{equation*}
As a function of $u$ (or $-u$ depending on conventions) and up to scaling, this is also called the Walsh transform of $x \mapsto \tr(v F(x))$.
The multiplicative Fourier transformation of $T^F$ is called the \emph{ultrametric integral transition matrix} $A^F$ of $F$, see~\cite{DBLP:conf/asiacrypt/BeyneV24,Beyne2023}:
\begin{equation*}
A^F = \UT \,T^F \,\UT^{-1}.
\end{equation*}
It fulfills the role of the correlation matrix in (ultrametric) integral cryptanalysis.
The coordinates of $A^F$ are directly related to the unique interpolating polynomial of $F$ with degree at most $q - 1$.

\begin{theorem}[{\cite[Theorem 5.8]{Beyne2023}}] \label{thm:reduction}
Let $F : \F_q \to \F_q$ be a function on the residue field $\F_q$ of $\QQ_p(\zeta_{q - 1}) \subset \K$. For all $j$ in $\{0, 1, \ldots, q - 1\}$, let $\sum_{i = 0}^{q - 1} a_{i, j}\, X^i$ be an interpolating polynomial of $F^j$ over $\F_q$.
The coordinates of the matrix $A^F$ are integral elements in $\QQ_p(\zeta_{q - 1})$ and
\begin{equation*}
A^F_{\lambda, \mu} \equiv a_{i, j} \pmod{p},
\end{equation*}
for $\lambda : x \mapsto \tau(x^j)$ and $\mu : x \mapsto \tau(x^i)$.
\end{theorem}
\begin{proof}
Every coordinate $A^F_{\lambda, \mu} = \lambda\big(T^F \,\mu^\dual\big)$ is an integral element in $\QQ_p(\zeta_{q - 1})$ because the coordinates of $T^F$, $\mu^\dual$ and $\lambda$ are all  in the ring of integers of $\QQ_p(\zeta_{q - 1})$. By the definition of $A^F$, the function $\lambda \circ F$ in $\K^{\F_q^n}$ is equal to
\begin{equation*}
T^{F^\dual} \lambda = \sum_{\mu} A^F_{\lambda, \mu} \; \mu,
\end{equation*}
where the sum is over all multiplicative characters of $\F_q$. That is, for every $\mu$ there exists an $i$ in $\{0, 1, \ldots, q - 1\}$ such that $\mu(x) = \tau(x^i)$ for all $x$ in $\F_q$. Evaluating $\lambda \circ F$ at $x$ and reducing modulo the maximal ideal $(p)$ of $\ZZ_p[\zeta_{q - 1}]$ gives
\begin{equation*}
F^j(x) \equiv \sum_{\substack{\mu\, :\, x \,\mapsto\, \tau(x^i)}} A^F_{\lambda, \mu}\;x^i \pmod{p}.
\end{equation*}
Since $F^j$ has a unique interpolating polynomial of degree at most $q - 1$ over $\F_q$, we can conclude that $A^F_{\lambda, \mu} \equiv a_{i, j} \pmod{p}$.
\end{proof}

The matrices $C^F$ and $A^F$ are both similar to $T^F$, hence similar to each other. More precisely, let $\TT = \FT \UT^{-1}$, then $C^F = \TT\,A^F\,\TT^{-1}$. In the following section, we analyze the change-of-basis matrices $\TT$ and $\TT^{-1}$ in detail.

\subsection{Bounds for Change-of-Basis}
The Gauss sum corresponding to an additive character $\chi$ and a multiplicative character $\lambda$ of $\F_q$ is defined as (the minus sign is a standard convention)
\begin{equation*}
G(\chi, \lambda) = -\sum_{x \,\in\, \F_q^\times} \chi(x)\,\lambda(x).
\end{equation*}
Such sums are well-understood. In particular, we have the following classical result on the divisibility of $G(\chi, 1/\lambda)$ by $p = -\pi^{p - 1}$, attributed to Stickelberger from the 19th century~\cite{stickelberger}.
By the algebraic degree of a multiplicative character $\lambda$, we mean the algebraic degree of the corresponding monomial. That is, if $\lambda(x) = \tau(x^k)$ with $0 \le k \le q - 1$, then $\adeg{\lambda} = \adeg{X^k} = s_p(k)$.

\begin{theorem}[Stickelberger's theorem for Gauss sums~{\cite[Chapter 1, Theorem 2.1]{lang2012cyclotomic}}] \label{thm:stickelberger}
For all nontrivial additive characters $\chi$ of $\F_q$ and all multiplicative characters $\lambda$ of $\F_q$ except $\lambda : x \mapsto \tau(x^{q - 1})$, the Gauss sum $G(\chi, 1/\lambda)$ satisfies
\begin{equation*}
\ord_p G(\chi, 1/\lambda) = \frac{\adeg{\lambda}}{p - 1}.
\end{equation*}
Furthermore, if $\lambda$ is $x \mapsto \tau(x^{q - 1})$, then $\ord_p G(\chi, 1/\lambda) = 0$.
\end{theorem}

The following lemma uses \Cref{thm:stickelberger} to bound the $p$-adic absolute values of the coordinates of the change-of-basis matrices $\TT = \FT \UT^{-1}$ and $\TT^{-1} = \UT \FT^{-1}$.

\begin{lemma} \label{lem:baseConversion}
For the matrix $\TT = \FT \UT^{-1}$ and its inverse $\TT^{-1}$, we have
\begin{equation*}
\ord_p \TT_{\chi, \lambda} =  \frac{\adeg{\lambda}}{p - 1},
\quad\mathrm{and}\quad
\ord_p \TT^{-1}_{\lambda, \chi} = -\frac{\adeg{\lambda}}{p - 1},
\end{equation*}
for every nontrivial character $\chi$ of the additive group $\F_q$ and every nontrivial character $\lambda$ of the multiplicative monoid $\F_q$.
Furthermore, $\TT_{1, \lambda} = 0$ for all nontrivial $\lambda$.
\end{lemma}
\begin{proof}
Throughout the proof, the multiplicative character $\lambda^* : x \mapsto \tau(x^{q - 1})$ will be a special case.
By the definition of $\UT$ and $\FT$, we have $\FT^\dual\,\delta^\chi = \chi$ and $\UT^{-1} \delta_\lambda = \lambda^\dual$. Hence,
\begin{align*}
\TT_{\chi, \lambda} = \delta^\chi\Big(\FT\,\UT^{-1}\delta_{\lambda}\Big) = \chi(\lambda^\dual)
&= -\delta_\lambda(\lambda^*) + \frac{1}{q - 1} \sum_{x \in \F_q^\times} \chi(x)/\lambda(x) \\
&= -\delta_\lambda(\lambda^*) - \frac{G(\chi, 1/\lambda)}{q - 1}.
\end{align*}
Note that $\TT_{\chi,\lambda^*} = - 1 -1 / (q - 1) = -q/(q-1)$, so $\ord_p \TT_{\chi,\lambda^*} = n = \adeg{\lambda^*}/(p-1)$.
If $\lambda \neq \lambda^*$, then \Cref{thm:stickelberger} yields
\begin{equation*}
\ord_p \TT_{\chi, \lambda} = \frac{\adeg{\lambda}}{p - 1}.
\end{equation*}
The result for $\chi = 1$ follows by a similar case distinction between $\lambda = \lambda^*$ and $\lambda \neq \lambda^*$.
The proof for $\TT^{-1}$ follows the same argument as above. Since $\UT^\dual \delta^\lambda = \lambda$ and $\FT^{-1} \delta_\chi = \chi^\dual$, we have
\begin{equation*}
\TT^{-1}_{\lambda, \chi} = \delta^\lambda\Big(\UT \FT^{-1}\,\delta_{\chi}\Big) = \lambda(\chi^\dual) = \frac{1}{q} \sum_{x \in \F_q^\times} \lambda(x) / \chi(x) = -\frac{G(1/\chi, \lambda)}{q}.
\end{equation*}
If $\lambda = \lambda^*$, then $\ord_p \TT^{-1}_{\lambda^*, \chi} = 0 - n = -\adeg{\lambda} /(p - 1)$.
Otherwise, using \Cref{thm:stickelberger} we get
\begin{equation*}
\ord_p \TT^{-1}_{\lambda, \chi} = \frac{\adeg{1/\lambda}}{p - 1} - n = \frac{n(p - 1) - \adeg{\lambda}}{p - 1} - n = -\frac{\adeg{\lambda}}{p - 1}.
\end{equation*}
In the second equality, we use the fact that $\adeg{1/\lambda} = n(p - 1) - \adeg{\lambda}$ for nontrivial $\lambda \neq \lambda^*$.
\end{proof}

\Cref{lem:baseConversion} implies the following theorem that bounds above the $p$-adic absolute value of the coordinates of $A^F$ in terms of those of $C^F$.
In \Cref{subsec:degreeBound}, it will be shown that the degree bound for planar functions is a special case of \Cref{lem:baseConversionBound}.

\begin{lemma} \label{lem:baseConversionBound}
Let $F$ be a function on $\F_q$. For all nontrivial multiplicative characters $\lambda$ and $\mu$ of $\F_q$, we have
\begin{equation*}
\ord_p A^F_{\mu, \lambda} \ge \frac{\adeg{\lambda} - \adeg{\mu}}{p - 1} + \min_{\chi \neq 1, \psi \neq 1} \ord_p C^F_{\psi, \chi},
\end{equation*}
where the minimum is over all nontrivial additive characters of $\F_q$.
\end{lemma}
\begin{proof}
The result follows from the ultrametric triangle inequality. More precisely,
\begin{equation*}
\ord_p A^F_{\mu, \lambda} = \ord_p \big(\TT^{-1}\,C^F\,\TT\big)_{\mu, \lambda}\ge \min_{\chi, \psi}\left( \ord_p \TT^{-1}_{\mu, \psi} + \ord_p C^F_{\psi, \chi} + \ord_p \TT_{\chi, \lambda}\right).
\end{equation*}
It can be assumed that $\chi \neq 1$ since $\TT_{1, \lambda} = 0$ for nontrivial $\lambda$.
Furthermore, $C^F_{1, \chi} = 0$ for $\chi \neq 1$. Hence, it can also be assumed that $\psi \neq 1$.
Applying \Cref{lem:baseConversion} to the first and last term in the sum concludes the proof.
\end{proof}

\subsection{Degree Bound} \label{subsec:degreeBound}
The field $\K$ has an automorphism $x \mapsto \overline{x}$ defined by $\zeta_p \mapsto 1/\zeta_p$. Given an embedding of $\K$ into $\CC$, it corresponds to complex conjugation.
Note that $x$ and $\overline x$ always have the same $p$-adic valuation.
Recall that $F$ is a planar function on $\F_q$ if and only if $x \mapsto \tr(vF(x))$ is bent for all nonzero $v$ in $\F_q$ (see~\cite{carlet2001generalized,DBLP:journals/dcc/Pott16}).
Equivalently, for all nontrivial additive characters $\chi$ and $\psi$ of $\F_q$,
\begin{equation*}
\overline{C^F_{\psi, \chi}} C^F_{\psi, \chi} = 1/q.
\end{equation*}
It terms of $p$-adic valuations, we have $\ord_p \smash{C^F_{\psi, \chi}} = -n/2$. \Cref{lem:baseConversionBound} can now be applied to bound above the $p$-adic absolute value of the coordinates of $A^F$.
This upper bound leads to \Cref{thm:degreeBoundPlanar} below.

\degreeBoundPlanar*
\begin{proof}
Let $\mu(x) = \tau(x^i)$ and $\lambda(x) = \tau(x^j)$. From \Cref{lem:baseConversionBound}, we find that
\begin{equation*}
\ord_p A^F_{\mu, \lambda} \ge \frac{\wt(j) - \wt(i)}{p - 1} - n/2 = \frac{\wt(j) - \wt(i)- (p - 1)n/2}{p - 1}.
\end{equation*}
By \Cref{thm:reduction}, $A^F_{\mu, \lambda}$ is an element of $\ZZ_p[\zeta_{q - 1}]$. Hence, it reduces to zero in the residue field $\F_q$ whenever $\ord_p A^F_{\mu, \lambda} > 0$. This implies that $x^j$ can only be a monomial in the unique interpolating polynomial of degree at most $q - 1$ for $x \mapsto F(x)^i$ if
\begin{equation*}
\wt(j) - \wt(i) - (p - 1) n/2 \le 0.
\end{equation*}
Hence, the algebraic degree of any monomial with nonzero coefficients in the unique interpolating polynomial of degree at most $q - 1$ for $x \mapsto F(x)^i$ is at most $(p - 1) n / 2 + \wt(i)$. If $G(x) = \sum_{i = 0}^{q - 1} a_i\,x^i$, then $G(F(x)) = \sum_{i = 0}^{q - 1} a_i\,F^i(x)$. Since the algebraic degree of $x \mapsto F^i(x)$ is at most $(p - 1)n/2 + \wt(i)$, the algebraic degree of $G \circ F$ is bounded above by 
\begin{equation*}
\adeg{G\circ F} \le (p - 1) n/2 + \max_{a_i \neq 0} \wt(i) = (p - 1) n/2 + \adeg{G}.
\end{equation*}
The result follows by rearranging terms.
\end{proof}

If $F$ and $G$ are monomial functions, then \Cref{thm:degreeBoundPlanar} can be simplified as follows.
\degreeBoundPlanarMonomial*
\begin{proof}
Let $F(x) = x^d$ and $G(x) = x^e$. Since $(G \circ F)(x) = x^{e \star d}$, we have $\adeg{G \circ F} = \wt(e \star d)$.
The result follows from \Cref{thm:degreeBoundPlanar}.
\end{proof}

\section{Planar Polynomials and Monomials over \texorpdfstring{$\F_{p}$}{Fields of Prime Order} and \texorpdfstring{$\F_{p^2}$}{Finite Fields of Degree 2}}
\label{sec:classification_prime}
Using the degree bound stated in \Cref{thm:degreeBoundPlanar}, we again establish the classification of planar polynomials over fields of prime order and planar monomials over fields of  prime-square order. Those results were originally obtained in~\cite{DBLP:journals/dm/Gluck90,hiramine1989conjecture,ronyai1989planar} and~\cite{coulter2006classification}, respectively. For the classification of planar monomials over extensions of prime-order fields, one resorts to the following basic fact, which is immediate from the definition of planarity.

\begin{lemma}
\label{lem:restriction}
    Let $\mathbb{S}$ be a subfield of $\F_q$ and let $F$ be a planar function on $\F_q$ so that the coefficients of its interpolating polynomial are in $\mathbb{S}$. The restriction of $F$ to $\mathbb{S}$ is planar on $\mathbb{S}$. 
\end{lemma}

The second crucial ingredient in the classification results from~\cite{DBLP:journals/ffa/BergmanCV22,coulter2006classification,DBLP:journals/ffa/CoulterL12} (the proof in~\cite{hiramine1989conjecture} also uses similar ideas) is Hermite's criterion for permutation polynomials.

\begin{lemma}[Hermite's Criterion~{\cite[Theorem 7.4]{Lidl_Niederreiter_1996}}]
    \label{lem:Hermite} A polynomial $Q$ in $\F_q[X]$ induces a permutation on $\F_q$ under evaluation if and only if the following two conditions hold:
    \begin{enumerate}
        \item $Q$ has exactly one root in $\F_q$, and
        \item for every positive integer $t \not\equiv 0 \pmod{p}$ less than $q-1$, the reduction of $Q(X)^t$ modulo $X^q-X$ has degree less than $q-1$.
    \end{enumerate}
\end{lemma}
In the literature, the non-planarity of a polynomial $F$ is often established by constructing an integer $t$ such that Condition 2 is violated for $Q(X) = F(X+\alpha)-F(X)$ with non-zero $\alpha$.
If $F$ is a monomial, then multiplying $X$ by $\alpha$ shows that one can assume $\alpha = 1$ without loss of generality.
Our proof of the classification of planar polynomials over prime fields uses the fact that the degree of a polynomial is equal to its algebraic degree, which significantly simplifies the application of the degree bound.
\begin{corollary} \label{cor:planarPrimeField}
If $F$ is a planar function on a field of prime order $p \geq 13$, then $\adeg{F} = 2$.
\end{corollary}
\begin{proof}
Without loss of generality, we can assume that $F(x) = \sum_{i=0}^{d}a_i\,x^i$ with $a_0, \ldots, a_d$ in $\F_p$, $a_d = 1$ and $0 \leq d \leq p-1$. As affine functions are not planar, we have $d \geq 2$. 
\Cref{thm:degreeBoundPlanar} with $G$ the identity function implies that $d \le (p + 1) / 2$.
For $2 < d < (p+1)/2$, we will choose an exponent $e$ such that the degree bound in \Cref{thm:degreeBoundPlanar} is violated for $G : x \mapsto x^e$. 
That is, we construct an $e$ such that
\begin{equation*}
\adeg{G \circ F} - \adeg{G} > \frac{p - 1}{2}.
\end{equation*}
Let us take $e = \lfloor (p - 1)/d \rfloor$, so that $\adeg{G \circ F} = \adeg{F^e} = ed$ and $\adeg{G} = e$. This construction works for all $d$ for which $e(d - 1) > (p - 1) / 2$, that is,
\begin{equation*}
\left\lfloor\frac{p - 1}{d}\right\rfloor (d - 1) > \frac{p - 1}{2}.
\end{equation*}
This is true provided that $3 \le d \le (p - 1) / 2$. Indeed, the case of $d \in \{ (p-1)/2 -1, (p-1)/2\}$ is straightforward. For the other cases, using the inequality $\lfloor y \rfloor > y-1$, we obtain
\begin{equation*}
\left\lfloor\frac{p - 1}{d}\right\rfloor (d - 1) > p-1 - \left(\frac{p-1}{d} + d -1 \right)
\end{equation*}
with $(p-1)/d + d -1 \leq (p-1)/2$, provided that $3 \leq d \leq (p-1)/2-2$ since $p \ge 13$.

Hence, the only remaining cases are $d = 2$ and $d = (p + 1) / 2$. The latter case can be ruled out because the unique interpolating polynomial with degree at most $p - 1$ of $x \mapsto F(x+1) - F(x) - F(1)$ would have degree $(p-1)/2$. However, $(p-1)/2$ is a non-trivial divisor of $p-1$, so Hermite's criterion implies that this polynomial does not induce a permutation on $\F_p$ under evaluation (see, e.g.,~\cite[Corollary 7.5]{Lidl_Niederreiter_1996}). 
\end{proof}

Note that the cases $p \leq 11$ can be handled computationally. Indeed, one only needs to check the planarity of polynomials of degree at most $(p-1)/2$ and without linear or constant term. There are at most $11^4$ such polynomials.

The classification of planar monomials over fields of order $p^2$ is a consequence of the classification for prime-order fields, and the degree bound with $e = 1$.
The proof uses the fact that if a monomial function is planar on $\F_{p^2}$, then it must also be planar on $\F_p$ (see Lemma~\ref{lem:restriction}).

\begin{corollary}
\label{cor:primeSquared}
If $F$ is a planar monomial function on a field of order $p^2$, then $\adeg{F} = 2$.
\end{corollary}
\begin{proof}
By restricting $F : x \mapsto x^d$ to $\F_p$ and applying \Cref{cor:planarPrimeField}, we get $\wt(d) \equiv 2 \pmod{p - 1}$.
That is, $\wt(d) = 2 + k(p - 1)$ for some $k \ge 0$. However, the degree bound from \Cref{cor:degreeBoundPlanarMonomial} with $e = 1$ implies that $\wt(d) \le p$.
Hence, $k = 0$ and $\wt(d) = 2$.
\end{proof}

Corollary~\ref{cor:primeSquared} states that if $X^d$ is planar over $\F_{p^2}$, we have $d \equiv p^i + p^j \pmod {p^2-1}$ with $0 \leq i \leq j \leq 1$. From the classification of planar Dembowski-Ostrom monomials~\cite{DBLP:journals/dcc/CoulterM97}, it then follows that $X^d$ is planar over $\F_{p^2}$ if and only if $d \equiv 2p^i \pmod {p^2-1}$ with $i$ in $\{0,1\}$.

\begin{remark}
    The proof of \Cref{cor:primeSquared} only depends on the classification of planar monomials over prime fields and the degree bound for $e=1$. The latter result was already known from the bound of Hou in \cite[Proposition 4.4]{DBLP:journals/ffa/Hou04}. Note that $F$ is planar on $\F_q$ if and only if every component function $x \mapsto \tr(vF(x))$ with $v \in \F_q^\times$ is bent.
\end{remark}

\begin{remark}
    The proof of the recent classification of planar monomials over fields of order $p^3$ can be significantly simplified as well. Indeed, suppose $X^d$ is a planar monomial over $\F_q$ with $q = p^3$ and $1 \leq d \leq q-1$. \Cref{cor:planarPrimeField} and the degree bound with $e=1$ imply that $s_p(d) = 2 + k(p-1)$ for $k \in \{0,1\}$. In~\cite{DBLP:journals/ffa/BergmanCV22}, settling the case of $k=1$ boils down to handling 11 different choices of $d$ (listed on page 22) and proving the non-planarity of $X^d$. Using the degree bound to establish non-planarity is often much simpler than applying Hermite's criterion, as was done in~\cite{DBLP:journals/ffa/BergmanCV22}. For instance, the most complicated exponent to handle was $d = 1 + 2p + (p-2)p^2$. Applying the degree bound with $e = (p-3)/2 + 2p + 2p^2$ (assuming $p> 11$) directly establishes non-planarity, whereas the method in~\cite{DBLP:journals/ffa/BergmanCV22} based on Hermite's criterion needed a long case-analysis.
\end{remark}

\section{Classifying Planar Monomials over \texorpdfstring{$\F_{p^{2^k}}$ for $p>5$}{Finite Fields of Degree 2ᵏ}}
Suppose that $n = 2m$ for an integer $m \geq 2$.
As any planar monomial over $\F_{p^n}$ must also be planar over the subfield of order $p^m$, the core of our argument is to show that any planar monomial $X^d$ over $\F_{p^n}$ fulfilling $d \equiv 2p^j \pmod {p^m-1}$ for a non-negative integer $j$ necessarily fulfills the congruence $d \equiv 2p^i \pmod {p^n-1}$ for some non-negative integer $i$. The classification result for $n$ a power of two then follows from an inductive argument, as already remarked in~\cite{DBLP:journals/ffa/CoulterL12}.

For integers $d_0,\dots,d_{n-1}$,
we denote by $[d_0,d_1,\dots,d_{n-1}]_p$ the integer $d = \sum_{i=0}^{n-1} d_i \cdot p^i$. If each $d_i$ for $0 \leq i \leq n-1$ fulfills $0 \leq d_i < p$ and not all $d_i$ are equal to $p-1$, we write $d = (d_0,d_1,\dots,d_{n-1})_p$. In that case, $(d_0,d_1,\dots,d_{n-1})_p$ is the base-$p$ representation of $d$ and we have $\wt(d) = \sum_{i=0}^{n-1} d_i$. For an integer $x$, we will write $\comp{x} = p - 1 - x$. If $x$ is a digit in $\{0, 1, \ldots, p - 1\}$, then $\comp{x}$ is its complement.

The following lemma on the digits of $d$ with $d \equiv 2p^{m-1} \pmod {p^m-1}$ was established by Coulter and Lazebnik using elementary arithmetic.

\begin{lemma}[\cite{DBLP:journals/ffa/CoulterL12}]
\label{lem:coulter_lazebnik}
    Let $n = 2m$ and $d = (d_0,d_1,\dots,d_{n-1})_p$ such that $d$ is congruent to $2p^{m-1} \pmod {p^m-1}$. The tuple $(d_0 + d_m, d_1+d_{m+1},\dots,d_{m-1}+d_{n-1})$ is equal to one of the following:\begin{enumerate}
        \item $(0,\dots,0,0,2)$
        \item $(p-1,\dots,p-1,p-1,p+1)$
        \item $(0,\dots,0,0,p,p-1,\dots,p-1,p-1,1)$.
    \end{enumerate}
\end{lemma}

To show that $X^{d}$ with $d \equiv 2p^{m-1} \pmod {p^m-1}$ is not planar over $\F_{p^n}$, we only need to consider exponents $d$ so that the tuple $(d_0+d_m,d_1+d_{m+1},\dots,d_{m-1}+d_{n-1})$ falls into one of the above cases.

The only planar monomials belonging to case 1 are those equivalent to $X^2$, i.e., either $d_{m-1} = 2$ or $d_{n-1} = 2$, corresponding to the monomials $X^{2p^{m-1}}$ and $X^{2p^{n-1}}$ respectively. When $d_{m-1} = d_{n-1} = 1$, the monomial $X^d$ is not planar, because $n/\gcd(n-1-(m-1),n)=2$. Case 2 was also solved completely in~\cite{DBLP:journals/ffa/CoulterL12}, but the proof was quite technical. Using the degree bound from \Cref{cor:degreeBoundPlanarMonomial} with $e = 1$, it is immediate that case 2 does not contain any planar monomials. Indeed, we have $\wt(d) = m(p-1) + 2$, whence $\wt(e \star d) - \wt(e) = m(p-1)+1 > n(p-1)/2$. Case 3 is the most complicated and was only solved for $n \leq 4$.  Based on the degree bound, we show that case 3 does not contain any planar monomial functions for $p>5$, establishing the classification of planar monomials over $\F_{p^{2^k}}$ for $p> 5$. 

In the following, we assume $n = 2m$. Let $r$ and $s$ be non-negative integers such that $r+s = m-2$. For $t, u_1,\dots,u_s$ in $\{0,\dots,p-1\}$ with $t\neq 0$, we define $D(t,u_1,\dots,u_s)$ to be the integer 
\begin{align}
\label{eq:general_form}
(\overbrace{0,\dots,0}^{r}, t, \overbrace{u_1,\dots,u_s}^{s},1,\overbrace{0,\dots,0}^{r},p-t, \overbrace{\comp{u}_1,\dots,\comp{u}_s}^s,0)_p,
\end{align} 
where $\comp{u}_i = p - 1 - u_i$.
Any $d$ corresponding to case 3 is of this form or a cyclic shift of it, because we can assume without loss of generality that the rightmost digit is a zero (in case it is a one, we can apply a cyclic shift). To prove our second main result (Theorem~\ref{thm:main_power2}), we will prove the following lemma. It implies that the only planar monomials that fall into cases 1, 2 or 3 are of the form $X^{2p^i}$. 

\begin{lemma}
\label{lem:power_2}
Suppose that $p>5$, $n$ is even, and let $s$ be a non-negative integer. For all $t$, $u_1,\dots,u_s$ in $\{0,\dots,p-1\}$ with $t \neq 0$, the monomial $X^{D(t,u_1,\dots,u_s)}$ is not planar over $\F_{p^n}$.
\end{lemma}
The main result then follows by an inductive argument. We will prove Lemma~\ref{lem:power_2} by showing that, for each $d = D(t,u_1,\dots,u_s)$, there exists an $e$ in $\{1,\dots,p^n-2\}$ such that $ed \equiv (r_0,\dots,r_{n-1})_p \not\equiv 0\pmod{p^n-1}$ (so that $e \star d = (r_0,\dots,r_{n-1})_p$) and $\wt(e \star d) - \wt(e) \geq m(p-1) + 1$. As there is no universal choice for $e$,
the proof is quite technical and split into several cases, depending on the values of the digits $t$ and $u_1,\dots,u_s$ of $d$. To find suitable candidates of $e$, we first conducted computer experiments. However, our proof is completely verifiable by hand. We start by ruling out all but a finite number of cases (for fixed dimension $n$), independently of $p$.

\subsection{Proof of Lemma~\ref{lem:power_2} up to a Finite Number of Cases for Fixed \texorpdfstring{$n$}{n}}
We start with the simplest case, where not all of $t,u_1,\dots,u_s$ are in $\{0,1,p-2,p-1\}$. Note that we do not exclude the case of $p=5$ here.

\begin{lemma} \label{lem:power_2_finite}
Suppose $p>3$, $n = 2m$ with $m \geq 2$ and let $s$ be a non-negative integer. If $d = D(t,u_1,\dots,u_s)$ with $t, u_1,\dots,u_s$ in $\{0,\dots,p-1\}$ and $t\neq 0$ such that $t \notin\{1,p-1\}$ or $u_j \notin \{0,1,p-2,p-1\}$ for some $j$ in $\{1,\dots,s\}$, then there exists an  $e$ in $\{1,\dots,p^n-1\}$ such that $\wt(e \star d) - \wt(e) > m(p-1)$.
\end{lemma}
\begin{proof}
Let $r = m-2 - s$.
We will multiply $d$ by $e = hp^m + h + 1$ for various choices of $h$ with $1 \leq h < p^m$. First, note that multiplying $d$ by $hp^m$ yields
\[[\overbrace{0,\dots,0}^{r},h(p-t),\overbrace{h\comp{u}_1,\dots,h\comp{u}_s}^{s},0,\overbrace{0,\dots,0}^{r},ht,\overbrace{hu_1,\dots,hu_s}^{s},h]_p \pmod {p^n-1}.\]
The product of $d$ and $h+1$ is congruent to
\begin{align*}
[&\overbrace{0,\dots,0}^{r},ht+t,\;\;\;\;\;\;\;\;\;\;\;\;\,\overbrace{hu_1+u_1,\dots,hu_s+u_s}^{s},h+1, \\
&\overbrace{0,\dots,0}^r,(h+1)(p-t),\overbrace{h\comp{u}_1 + \comp{u}_1,\dots,h\comp{u}_s + \comp{u}_s}^{s},0]_p.
\end{align*}
By adding up both results, we find that $ed$ with $e = h p^m + h + 1$ is congruent to  (mod $p^n - 1$)
\begin{align}
\label{eq:de}
    [\overbrace{0,\dots,0}^r,t,\overbrace{u_1,\dots,u_s}^s,2h+1,\overbrace{0,\dots,0}^r,p-t,\overbrace{\comp{u}_1,\dots,\comp{u}_s}^s,2h]_p.
\end{align}

We first consider the case that $t \notin \{1,p-1\}$.
Let $h = p^{r + 1} - 1$. As $1 \leq h+1 <p^m$, we have $\wt(e) = \wt(h+1) + \wt(h)$ with $\wt(h+1) = 1$ and $\wt(h) = (r+1)(p-1)$, so $\wt(e) = (r+1)(p-1) + 1$.
Moreover, $2h + 1 = (p^{r + 1} - 1) + p^{r + 1} = (p - 1, \ldots, p - 1, 1, 0, \ldots, 0)_p$ with $r + 2$ non-zero digits.
Hence, $ed$ is congruent to
\[[\overbrace{p-1,\dots,p-1}^r,t+1,\overbrace{u_1,\dots,u_s}^s,\overbrace{p-1,\dots,p-1}^{r+1},p-t+1,\overbrace{\comp{u}_1,\dots,\comp{u}_s}^s,p-2]_p.\]
Since $t \notin \{1,p-1\}$, all elements of the above tuple are in the range $\{0, 1, \ldots, p - 1\}$. Hence,
\[e \star d = (\overbrace{p-1,\dots,p-1}^r,t+1,\overbrace{u_1,\dots,u_s}^s,\overbrace{p-1,\dots,p-1}^{r+1},p-t+1,\overbrace{\comp{u}_1,\dots,\comp{u}_s}^s,p-2)_p.\]
For the sum of base $p$ digits, using $2r + s + 1 = m + (r + 1) - 2$, we have
\begin{equation*}
\wt(e \star d) = (2r+s+1)(p-1) + 2p = m(p - 1) + (r + 1)(p - 1) + 2 > m(p - 1) + \wt(e).
\end{equation*}
This establishes the first case.

Let us now consider the case where there exists an index $j$ in $\{1,\dots,s\}$ such that $u_j \notin \{0,1,p-2,p-1\}$. Let $h = p^r - 1 + ((p-1)/2-1)p^r + p^{r+j+1}$. As before, it holds that $\wt(e) = \wt(h+1) + \wt(h)$. Since $\wt(h + 1) = (p - 1) / 2 + 1$ and $\wt(h) = r(p - 1) + (p - 1)/2$, this implies $\wt(e) = (r + 1)(p - 1) + 1$.
Furthermore, $2h + 1 = 2p^r - 1 + (p - 3)p^r + 2 p^{r + j + 1} = (p - 2)p^r + p^r - 1 + 2p^{r + j + 1}$ so
\begin{equation*}
2h + 1 = (\overbrace{p - 1, \ldots, p - 1}^r, p - 2, \overbrace{0, \ldots, 0}^j, 2, \overbrace{0, \ldots, 0}^{s - j})_p.
\end{equation*}
We now substitute $h$ in~\eqref{eq:de} and distinguish between the cases $r>0$ and $r= 0$. By the choice of $j$, both of $u_j+2$ and $p-1-u_j+2$ are in $\{0,\dots,p-1\}$. For $r>0$, we obtain
\begin{align*}
e \star d = (&\overbrace{p-1,\dots,p-1}^{r-1},p-2,\hphantom{p -\;\;} t,u_1,\dots,u_{j-1},u_j+2,u_{j+1},\dots,u_s,p-1,\\
    &\,p-1,\dots,p-1,p-2,p-t,\comp{u}_1,\dots,\comp{u}_{j-1},\comp{u}_{j}+2,
    \comp{u}_{j+1},\dots,\comp{u}_s,p-2)_p.
\end{align*}
For $r = 0$, we obtain that $e \star d$ is equal to
\begin{equation*}
    (t,u_1,\dots,u_{j-1},u_j+2,u_{j+1},\dots,u_s,p-2,
    p-t,\comp{u}_1,\dots,\comp{u}_{j-1},\comp{u}_{j}+2,
    \comp{u}_{j+1},\dots,\comp{u}_s,p-3)_p.
\end{equation*}
In both cases, the sum of the base $p$ digits is equal to $\wt(e \star d) = m(p-1)+ (r+1)(p-1)+2$. Since $\wt(e) = (r+1)(p-1)+1$ as shown above, the result follows.
\end{proof}

We have now established that, for a fixed dimension $n$, the number of cases left to check is independent of $p$. Indeed, $t$ can only take one of $1$, $p-1$, while each $u_j$ for $1 \leq j \leq s$ can be $0$, $1$, $p-2$, or $p-1$. 

\subsection{Proof of Lemma~\ref{lem:power_2} for the Remaining Cases}

The argument for settling the remaining cases is more technical and assumes $p \neq 5$. For our proof, we need the following technical lemma.

\begin{lemma}
\label{lem:technical}
Let $p > 3$. Let $s$ be a positive integer and $u_0,u_1,\dots,u_s \in \{0,1,p-2,p-1\}$. For $v = [2u_1-u_0,\dots,2u_{s-1}-u_{s-2},2u_s -u_{s-1}]_p \neq p^s-1$, there exists $\delta$ in $\{-1, 0, 1\}$ such that
\begin{equation*}
v + \delta\,p^s =  [\gamma_1,\dots,\gamma_s]_p,
\end{equation*}
with $\gamma_1, \ldots, \gamma_s$ in $\{0, \ldots, p - 1\}$ and $0 < u_s + \delta + 1 \le p - 1$.
Furthermore, if $u_s = 1$, $\delta = 1$ and $p > 5$, then there exists an index $j$ such that $\gamma_j \not\in \{0, 1, p - 2, p - 1\}$.
\end{lemma}
\begin{proof}
For every index $j$, it holds that $1-p \le 2u_j - u_{j-1} \le 2(p - 1)$.
Since $(p - 1)\sum_{i = 0}^{s - 1} p^i = p^s - 1$, it follows that $v \ge 1 - p^s \ge -p^s$ and $v \le 2(p^s - 1) \le 2p^s - 1$.
Hence, there exists a $\delta$ in $\{-1, 0, 1\}$ such that $0 \le v + \delta p^s \le p^s - 1$.

The prove the lower bound $u_s + \delta + 1 > 0$, it is sufficient to analyze the case $u_s = 0$.
If $u_s = 0$, then $2u_s - u_{s-1} \leq 0$, so $v$ is bounded above by $2(p - 1)\sum_{i=0}^{s-2} p^i = 2 p^{s-1} - 2 \leq p^s-1$. Hence, $\delta \neq -1$ and consequently $u_s + \delta + 1 > 0$ for all values of $u_s$.

For the upper bound $u_s + \delta + 1 \le p - 1$, it enough to consider $u_s$ in $\{p - 2, p - 1\}$. If $u_s = p-2$, then $2u_s - u_{s-1} \geq p-3$, so $v \ge (p-3)p^{s-1} + (1 - p)\sum_{i=0}^{s-2} p^i = (p-4)p^{s-1} + 1 \geq 0$. Hence, $\delta \neq 1$ and the bound $u_s + \delta + 1 \le p - 1$ is satisfied. For the case $u_s = p - 1$, we show that $\delta = -1$ when $v \neq p^s - 1$. From $v = \smash{2u_s p^{s - 1} - u_0 - (p - 2) \sum_{i = 0}^{s - 2} u_{i + 1} p^i}$, we see that in this case the value of $v$ is minimized for $u_0 = p - 2$ and $u_1 = u_2 = \ldots = u_s = p - 1$. Hence, $v \ge p + (p - 1)\sum_{i = 1}^{s - 1} p^i = p^s$. This implies that $\delta = -1$, from which the upper bound follows.

To prove the final claim, define $v_j = 2u_j - u_{j-1}$. If $u_s = 1$, then $v_s \in \{2,1,3-p,4-p\}$. More precisely, $v_s = 2 - u_{s - 1}$.
If $\delta = 1$, then it follows that $v_s \in \{3-p,4-p\}$. Indeed this is trivial if $s= 1$ and, for $s>1$, this follows from the fact that $-p^{s-1} \leq [v_1,\dots,v_{s-1}]_p \leq 2p^{s-1}-1$. Let now $j$ be the smallest index in $\{1,\dots,s\}$ with the property that $u_j = 1$ and $v_{j} \notin \{1,2\}$. As discussed before, this index $j$ exists and $v_{j} \in \{4-p,3-p\}$. In what follows, we assume $p\geq7$.

\paragraph{Case $v_{j} = 3 - p$.} If $j = 1$, then we have $\gamma_{j} = \gamma_1 = 3$ (and the $-p$ decreases the value $\gamma_2$ by 1 as a negative carry). If $j \geq 2$, then we have $\gamma_{j} = 3 +c$, where $c$ is the carry coming from $[v_1,\dots, v_{j-1}]_p$. As $-p^{j-1} \leq [v_1,\dots,v_{j-1}]_p \leq 2p^{j-1}-1$, the carry $c$ can only take the values $0,1,$ or $-1$. Hence, $\gamma_{j} \in \{2,3,4\}$. The result follows because $2,3,4 \notin \{p-2,p-1\}$ for $p \ge 7$.

\paragraph{Case $v_{j} = 4-p$.} Similarly as in the case above, if $j = 1$, then we have $\gamma_{j} = \gamma_1 = 4$ and we are done. Let us therefore assume $j \geq 2$. We have $\gamma_{j} = 4 +c$, where $c \in \{-1,0,1\}$ is the carry coming from $[v_1,\dots, v_{j-1}]_p$ (note that for $p \geq 11$, the proof would be finished here). 
As discussed above, if $v_{j} = 4-p$, then we have $u_{j-1} = p-2$, so
\begin{align*}
v_{j-1} &\in \{2(p-2), 2(p-2) - 1, 2(p-2) - (p-2), 2(p-2) - (p-1)\} \\
&= \{p+p-4, p + p-5,p-2, p-3\}.
\end{align*}
If $v_{j-1} \in \{p-2,p-3\}$, then we have $c = 0$, so $\gamma_{j} = 4$ and the result follows. In case $v_{j-1} = p+p-4$, we have $\gamma_{j-1} = p-4+c'$ for a carry $c' \in \{-1,0,1\}$, so $\gamma_{j-1} \in \{p-3,p-4,p-5\}$ and the result follows as well. Finally, suppose that $v_{j-1} = p+p-5$, so that $\gamma_{j-1} = p-5 + c'$ for a carry $c' \in \{-1,0,1\}$. If $j-1 = 1$, then we have $c'= 0$ (there is no carry), so $\gamma_{j-1} = p-5$ and the result follows as $p-5 \notin \{0,1\}$ if $p \geq 7$. In case of $j-1 \geq 2$, we must have $u_{j-2} = 1$, and thus $v_{j-2} \in \{1,2,3-p,4-p\}$. Since $j \ge 3$ was chosen as the least integer in $\{2,\dots,s\}$ with the property that $u_j = 1$ and $v_{j} \notin \{1,2\}$, we must have $v_{j-2} \in \{1,2\}$. In this case we have $c'=0$, and the result follows. 
\end{proof}

Now, we settle the remaining cases, assuming $p> 5$.
\begin{lemma}
\label{lem:remaining_p7}
Let $p>5$, $n = 2m$ with $m \geq 2$, and $s$ a non-negative integer. If we have $d = D(t,u_1,\dots,u_s)$ with $t$ in $\{1,p-1\}$ and $u_1,\ldots,u_s$ in $\{0,1,p-2,p-1\}$, then there exists an $e$ in $\{1,\dots,p^n-1\}$ such that $\wt(e \star d) - \wt(e) > m(p-1)$.
\end{lemma}
\begin{proof}
Let $r = m - 2 - s$.
We will multiply $d$ by $e = hp^m + h+1 + (p-1)p^m = (p-1+h)p^m + h+1$ for various choices of $h$ with $1 \leq h < p^{m+1}$.
The case $s = 0$ will be handled separately from the case $s > 0$.

\paragraph{Case $s=0$.} 
We first multiply $d$ by $(p-1)p^m$ and obtain (see \eqref{eq:general_form} for $d$)
\begin{equation*}
d (p-1)p^m \equiv [\overbrace{0,\dots,0}^r,t,p-1-t,\overbrace{0,\dots,0}^r,-t,p-1+t]_p \pmod{p^n-1}.
\end{equation*}
Let $e = hp^m + h+1 + (p-1)p^m$. By adding the term \eqref{eq:de} from the proof of \Cref{lem:power_2_finite}, we obtain 
\begin{equation*}
    ed \equiv [\overbrace{0,\dots,0}^r,2t, 2h-t+p,\overbrace{0,\dots,0}^r,p-2t, p-1+2h+t]_p \pmod{p^n-1}.
\end{equation*}
If $t = p-1$, then $2t = p + p-2$, so we can write $e \star d$ (mod $p^n-1$) as
\begin{equation*}
    [\overbrace{0,\dots,0}^r,p-2, 2h+2,\overbrace{0,\dots,0}^r,2, p-2+2h+p-1]_p.
\end{equation*}
In particular, whenever $t \in \{1,p-1\}$, we have
\begin{equation*}
    e \star d \equiv [\overbrace{0,\dots,0}^r,t', 2h-u_0+p,\overbrace{0,\dots,0}^r,p-t', p-1+2h+u_0]_p \pmod {p^n-1},
\end{equation*}
where $t' \in \{2,p-2\}$, and $u_0 = 1$ if $t' = 2$ and $u_0 = p - 2$ if $t' = p - 2$.

If $t' = 2$, let $h = p^{r + 1} - 1 - (p - 1)/2$. As $0 \leq h+1 <p^m$ and $0 \leq p-1+h <p^m$, we have $\wt(e) = \wt(h+1) + \wt(p-1+h)$. Since $\wt(h + 1) = r(p-1) + (p - 1)/2 + 1$ and $\wt(p - 1+ h) =  (p - 1)/2$, it follows that $\wt(e) = (r + 1)(p - 1) + 1$.
From $2h = p^{r + 1} - 1 + \sum_{i = 1}^r (p - 1)p^i$, we obtain
\begin{align*}
e \star d = [&\overbrace{p-1,\dots,p-1}^r,t'+1, p-1-u_0,\overbrace{p-1,\dots,p-1}^r,p-t'+1, p-2+u_0]_p \\
= (&\overbrace{p-1,\dots,p-1}^r,3, p-2,\overbrace{p-1,\dots,p-1}^r,p-1, p-1)_p.
\end{align*}
Hence, $\wt(e \star d) = 2(r+1)(p-1) + p+1 = m(p-1) + (r+1)(p-1) + 2 > m(p - 1) + \wt(e)$.

If $t' = p-2$, let $h = p^{r + 1} - (p - 1)$. We then have $\wt(e) = \wt(h+1) + \wt(p-1+h) = 2 + r(p-1) + 1 = r(p-1) + 3$ and $2h = 2 - p + p^{r+1} + \sum_{i=1}^r(p-1)p^i$. If $r>0$, we obtain
\begin{align*}
    e \star d \equiv [&p-2,\overbrace{p-1,\dots,p-1}^{r-1},t'+1, 2-u_0+p,\\
    &p-2,\overbrace{p-1,\dots,p-1}^{r-1},p-t'+1, p+1+u_0]_p \\
    \equiv [&p-2,\overbrace{p-1,\dots,p-1}^{r-1},p-1, 4,p-2,\overbrace{p-1,\dots,p-1}^{r-1},3, p+p-1]_p \\
    \equiv (&p-1,\overbrace{p-1,\dots,p-1}^{r-1},p-1, 4,p-2,\overbrace{p-1,\dots,p-1}^{r-1},3,p-1)_p \pmod {p^n-1},
\end{align*}
with $\wt(e \star d) = 2(r+1)(p-1)+6 = (r+2)(p-1) + r(p-1)+6 > m(p - 1) + \wt(e)$. If $r = 0$, then we obtain (using $t' = p - 2$ and $u_0 = p - 2$) that $e \star d$ is congruent to
\begin{align*}
    [t', 2-u_0+p,p-t', p+1+u_0]_p &\equiv [p-2, 4,2, p+p-1]_p \\
    &\equiv (p-1, 4,2, p-1)_p \pmod {p^n-1}.
\end{align*}
Hence, $\wt(e \star d) = (r+2)(p-1) + r(p-1) + 6 > m(p - 1) + \wt(e)$.

\paragraph{Case $s>0$.} 
Multiplying $d$ by $(p-1)p^m$ modulo $p^n - 1$ yields
\begin{align*}
[&\overbrace{0,\dots,0}^r,t,u_1-t,\overbrace{u_2-u_1,u_3-u_2,\dots,u_s-u_{s-1}}^{s-1},p-1-u_s,\\
&\overbrace{0,\dots,0}^r,-t,-(u_1-t),\overbrace{-(u_2-u_1),-(u_3-u_2),\dots,-(u_s-u_{s-1})}^{s-1},p-1+u_s]_p.    
\end{align*}
Let $e = hp^m + h+1 + (p-1)p^m$. By addition with term \eqref{eq:de} from the proof of \Cref{lem:power_2_finite}, we obtain that $e \star d$ is congruent to
\begin{align*}
    [&\overbrace{0,\dots,0}^r,\phantom{p-}\;\,2t,2u_1-t,\overbrace{2u_2-u_1,\ldots,2u_s-u_{s-1}}^{s-1}, 2h-u_s+p,\\ &\,0,\dots,0,p-2t, \overline{2u_1-t},\overline{2u_2-u_1},\dots,\overline{2u_s-u_{s-1}},p-1+2h+u_s]_p.
\end{align*}
If $t = p-1$, then we have $2t = p + p-2$, so we can write $e \star d$ (mod $p^n-1$) as
\begin{align*}
    [&\overbrace{0,\dots,0}^r,\hphantom{p-()}p-2,2u_1-(p-2),\overbrace{2u_2-u_1,\dots,2u_s-u_{s-1}}^{s-1}, 2h-u_s+p,\\ &\,0,\dots,0,\,p-(p-2), \overline{2u_1-(p-2)},\overline{2u_2-u_1},\dots,\overline{2u_s-u_{s-1}},p-1+2h+u_s]_p.
\end{align*}
Hence, in any case of $t \in \{1,p-1\}$, we have
\begin{align*}
e \star d \equiv [&\overbrace{0,\dots,0}^r,\hphantom{p-}\;\,t',\overbrace{2u_1-u_0,\ldots,2u_s-u_{s-1}}^{s}, 2h-u_s+p,\\ &0,\dots,0,p-t', \overline{2u_1-u_0},\ldots,\overline{2u_s-u_{s-1}},p-1+2h+u_s]_p \pmod {p^n-1},
\end{align*}
where $t' \in \{2,p-2\}$ and $u_0 = 1$ if $t' = 2$ and $u_0 = p - 2$ if $t' = p - 2$. Define $v = [2u_1-u_0,\dots,2u_s - u_{s-1}]_p$. Note that $v \neq p^s-1$, as $u_0 \neq p-1$.
By \Cref{lem:technical}, there exists a $\delta$ in $\{-1, 0, 1\}$ such that $v = (\gamma_1, \ldots, \gamma_s)_p - \delta p^s$ with $\gamma_1, \ldots, \gamma_s$ in $\{0,\dots,p-1\}$. Hence, $e \star d$ is congruent to
\begin{equation*}
[\overbrace{0,\dots,0}^r,t',\overbrace{\gamma_1,\dots,\gamma_s}^s, 2h-u_s+p-\delta, \overbrace{0,\dots,0}^r,p-t', \overbrace{\comp{\gamma}_1,\dots,\comp{\gamma}_s}^s,p-1+2h+u_s+\delta]_p.
\end{equation*}
We finish the proof using a case distinction between three subcases corresponding to $u_s \notin \{0, 1\}$ and $r = 0$, $u_s \notin \{0, 1\}$ and $r > 0$, and $u_s \in \{0, 1\}$.

\paragraph{Subcase $u_s \notin \{0,1\}$ and $r=0$.} We choose $h = 1$ so that $\wt(e) = 3$. We have
\begin{equation}
\label{eq:fail_p5}
e \star d = [t'+1,\overbrace{\gamma_1,\dots,\gamma_s}^s, p-u_s+2-\delta, p-t', \overbrace{\comp{\gamma}_1,\dots,\comp{\gamma}_s}^s,1+u_s+\delta]_p.
\end{equation}
By assumption, $u_s \in \{p-2,p-1\}$ and $p>5$, so $0 \leq p-u_s+2 - \delta \leq p-1$.  Moreover, by \Cref{lem:technical}, $0 \leq 1+u_s + \delta \leq p-1$. Hence, $\wt(e \star d) = (s+2)(p-1)+6 = m(p-1)+6 > m(p - 1) + \wt(e)$.

\paragraph{Subcase $u_s \notin \{0,1\}$ and $r>0$.}
Let $h = ((p - 1)/2)\,p^r - (p - 1)$. 
Since $\wt(h + 1) = r(p - 1) - (p - 1) / 2 + 1$ and $\wt(p - 1 + h) = (p - 1)/2$, we have $\wt(e) = \wt(h+1) + \wt(p-1+h) = r(p-1) + 1$. Furthermore, $2h = (p - 1) p^r - 2(p - 1) = (p-2)p^r + 2 - p + \sum_{i=1}^{r-1}(p-1)p^i$. Hence, $e \star d$ is
\begin{align*}
[&\overbrace{p-1,\dots,p-1}^{r-1},p-2,\hphantom{p-}\;\,t',\overbrace{\gamma_1,\dots,\gamma_s}^s,2-u_s-\delta, \\ 
&\overbrace{p-1,\dots,p-1}^{r-1},p-2,p-t', \overbrace{\comp{\gamma}_1,\dots,\comp{\gamma}_s}^s,1+u_s+\delta]_p.
\end{align*}
To determine the sum of base-$p$ digits of $e\star d$, we use the same argument as for the case $r=0$ above. If $r-1 > 0$, then $e \star d$ is equal to
\begin{align*}
&(\overbrace{p-1,\dots,p-1}^{r-1},\;\;\;\;\;\;\;\;\,p-2,t',\overbrace{\gamma_1,\dots,\gamma_s}^s, p-u_s+2-\delta,\\ 
&p-2,\overbrace{p-1,\dots,p-1}^{r-2},p-2,p-t', \overbrace{\comp{\gamma}_1,\dots,\comp{\gamma}_s}^s,1+u_s+\delta)_p.
\end{align*}
If $r-1 = 0$, then $e \star d$ equals
\begin{equation*}
(p-2,t',\overbrace{\gamma_1,\dots,\gamma_s}^s,p-u_s+2-\delta, p-3,p-t', \overbrace{\comp{\gamma}_1,\dots,\comp{\gamma}_s}^s,1+u_s+\delta)_p.
\end{equation*}
In both cases, we obtain $\wt(e\star d) = m(p-1) + r(p-1)+2 > m(p - 1) + \wt(e)$. 

\paragraph{Subcase $u_s \in \{0, 1\}$.} Let $z \in \{1,\dots,s+1\}$ and $h = (p^{r + 1 + z} - 1) + (p^{r + 1} - 1) - (p - 1)/2$. Equivalently, $h = p^{r + 1 + z} - 1 + (p - 1)/2 + \sum_{i = 1}^r (p - 1)p^i$.
Let us first analyze the sum of base-$p$ digits of $e$. If $z \neq s+1$, then we have $0 \leq p-1+h <p^m$ and $0 \leq h+1 < p^m$ with $\wt(h+1) = (p-1)/2 + r(p-1) + 1$ and $\wt(p-1+h) = (p - 1)/2$, yielding $\wt(e) = (r+1)(p-1) +1$. If $z = s+1$, then $r+1+z = m$, so $(p-1+h)p^m + h+1$ is congruent to
\begin{align*}
&\frac{p-1}{2} + \sum_{i=1}^r(p-1)p^i + p^m + \left(\frac{p-1}{2}-2\right)p^m + p^{r+1+m} + 1 \\
\equiv &\left(\frac{p-1}{2}+1\right) + \sum_{i=1}^r(p-1)p^i + \left(\frac{p-1}{2}-1\right)p^m + p^{r+1+m} \pmod{p^n - 1}.
\end{align*}
We take $e$ with $1 \leq e \leq p^n-1$ congruent to $(p-1+h)p^m + h+1$, so that
$\wt(e) = (r+1)(p-1) + 1$ as well. 
Multiplying $h$ by $2$, we obtain $2h = p^{r+1} + 2p^{r+1+z} -3 + \sum_{i=1}^r(p-1)p^i$.
We now consider the cases $(u_s,\delta) \neq (1,1)$  and $(u_s,\delta) = (1,1)$ separately.

In case $(u_s,\delta) \neq (1,1)$, choosing $z = s+1$ yields
\begin{align*}
e \star d \equiv [&\overbrace{p-1,\dots,p-1}^r,\hphantom{p-}\;\,t'+1,\overbrace{\gamma_1,\dots,\gamma_s}^s, p-u_s-1-\delta,\\ &p-1,\dots,p-1,p-t'+1, \comp{\gamma}_1,\dots,\comp{\gamma}_s,p-2+u_s+\delta]_p \pmod {p^n-1}. 
\end{align*}
From \Cref{lem:technical}, we have $0 \le u_s + \delta < p - 1$ (note that $v \neq p^s - 1$ as $u_0 \neq p - 1$). Hence, $0 < p - 1 - u_s - \delta \le p - 1$.
Furthermore, by assumption $u_s \in \{0, 1\}$ and $(u_s, \delta) \neq (1, 1)$, so $0 \le p - 2 + u_s + \delta \le p - 1$. Hence,
\begin{align*}
e \star d = (&\overbrace{p-1,\dots,p-1}^r,\phantom{p-}\;\, t'+1,\overbrace{\gamma_1,\dots,\gamma_s}^s, p-u_s-1-\delta,\\ &\,p-1,\dots,p-1,p-t'+1, \comp{\gamma}_1,\dots,\comp{\gamma}_s,p-2+u_s+\delta)_p,
\end{align*}
with $\wt(e \star d) = (r+s+2)(p-1)+ (r+1)(p-1) + 2 > m(p - 1)+ \wt(e)$.

Finally, let us consider the case $(u_s,\delta) = (1,1)$. By \Cref{lem:technical}, there exists an index $j$ in $\{1,\dots,s\}$ such that $\gamma_j \notin\{0,1,p-2,p-1\}$. By choosing $z = j$, we obtain
\begin{align*}
e \star d = [&\overbrace{p-1,\dots,p-1}^r,\hphantom{p - }\,\;t'+1,\gamma_1,\dots,\gamma_{z-1},\gamma_z + 2, \gamma_{z+1},\dots,\gamma_s, p-u_s-3-\delta,\\ &\,p-1,\dots,p-1,p-t'+1, \comp{\gamma}_1,\dots,\comp{\gamma}_{z-1},\comp{\gamma}_{z}+2,\comp{\gamma}_{z+1},\dots,\comp{\gamma}_s,p-4+u_s+\delta]_p \\
= (&p-1,\dots,p-1,\phantom{p - }\,\;t'+1,\gamma_1,\dots,\gamma_{z-1},\gamma_z + 2, \gamma_{z+1},\dots,\gamma_s, p-5,\\ &p-1,\dots,p-1,p-t'+1, \comp{\gamma}_1,\dots,\comp{\gamma}_{z-1},\comp{\gamma}_{z}+2, \comp{\gamma}_{z+1},\dots,\comp{\gamma}_s,p-2)_p,
\end{align*}
with $\wt(e \star d) = (r+s+2)(p-1)+ (r+1)(p-1) + 2 > m(p - 1) + \wt(e)$.
\end{proof}

For $p=5$, the only situation in which the construction of $e$ in the proof of \Cref{lem:remaining_p7} might fail is when $s>0$ and $u_s \in \{1,p-2\}$. Indeed, if $u_s = p-2$, then $p-u_s+2 - \delta$ in expression~\eqref{eq:fail_p5} can be equal to $p$ when $\delta = -1$ (then $u_{s-1} \in\{0,1\}$). If $u_s = 1$ and $\delta = 1$ (i.e., $u_s = 1, u_{s-1} \in \{p-2,p-1\}$), then \Cref{lem:technical} does not guarantee existence of an index $j$ such that $\gamma_j \notin \{0,1,p-2,p-1\}$ if $p=5$. 

\subsection{Proof of the Main Result}
As already explained by Coulter and Lazebnik in~\cite{DBLP:journals/ffa/CoulterL12}, an inductive argument now yields the classification of planar monomials over $\smash{\F_{p^{2^k}}}$ for $p>5$.
\powerTwo*
\begin{proof}
Since every monomial $X^{2p^i}$ is planar over $\smash{\F_{p^{2^k}}}$, we only need to show that any planar monomial $X^d$ over $\smash{\F_{p^{2^k}}}$ is of the form $d \equiv 2p^i \pmod {p^{2^k}-1}$.
The proof is by induction on $k$. The statement holds for $k \in \{0,1\}$, see~\cite{johnson1987projective,coulter2006classification} and our new proofs in \Cref{sec:classification_prime}. Suppose that the statement holds for $k-1$ and define $m = 2^{k-1}$ and $n = 2m = 2^k$. If $X^d$ is planar over $\F_{p^n}$, then it must be planar over the subfield $\F_{p^m}$. The induction hypothesis yields $d \equiv 2p^j \pmod{p^m-1}$ for some non-negative integer $j$. Hence, a cyclic shift of $(d_0+d_m,d_1+d_{m+1},\dots,d_{m-1}+d_{n-1})$ must be in one of the three cases listed in \Cref{lem:coulter_lazebnik}, where $d = (d_0,d_1,\dots,d_{n-1})_p$. The only planar monomials in case 1 have exponent $d = 2p^i$ for some non-negative integer $i$. As discussed earlier, case 2 does not contain planar monomials. Finally, \Cref{lem:power_2} implies that case 3 does not contain planar monomials either.
\end{proof}

\section{A Conjecture on the Base-\texorpdfstring{$p$}{p} Digit Sum of Integers}
Based on computations, we raise the following conjecture.

\begin{conjecture}
\label{conj:extended}
For all positive integers $d \le q - 1$ with $d \not\equiv 1 \pmod{p-1}$ and $d \not\equiv\lfloor (p + 1)/2 \rfloor \pmod{p-1}$, there exists a positive integer $e \le q - 1$ such that \begin{align}\label{eq:weight}\wt(e \star d) - \wt(e) > \frac{n(p-1)}{2},\end{align}
    unless
    \begin{enumerate}[(i)]
        \item $\wt(d) = 2$, or
        \item $p=5$, $n$ odd, and $d=5^j (5^i+1)/3$ for some non-negative integers $i,j$.
    \end{enumerate}
\end{conjecture}

Note that if either $d \equiv 1 \pmod{p-1}$ or $d \equiv \lfloor (p + 1)/2 \rfloor \pmod{p-1}$, then there are many $d$ for which no positive $e \leq q-1$ exists such that inequality~\eqref{eq:weight} holds. The exceptional case $d = 5^j (5^i + 1) / 3$ for $p = 5$ resembles the counterexample $d = 3^j (3^i + 1)/2$ for $p = 3$ to the original Dembowski-Ostrom conjecture that was found by Coulter and Matthews~\cite{DBLP:journals/dcc/CoulterM97}.

\Cref{conj:extended} has been proven for $n=1$; see the proof of \Cref{cor:planarPrimeField}. Moreover, we computationally verified it for all $p \leq 800$ if $n = 2$, for all $p \leq 400$ if $n = 3$, and, if $n \geq 4$, for all $(p,n)$ with $p^n \leq 10^9$. 
\begin{remark}
In the computational verification, we included non-prime $p$ as well. The only counterexamples we found were for the case $p=9$ with $d$ of the form $d = 3p^i$.
\end{remark}

\paragraph{Some special cases of the conjecture.}
A proof of the implication in \Cref{conj:extended} for any special case of $(q, d)$ shows that the monomial $X^d$ is not planar over $\F_{q}$. In particular, \Cref{conj:extended} for the special case $d \equiv 2 \pmod {p-1}$ implies the classification of planar monomials over finite fields of characteristic $p>5$.

If $s_p(d) > n(p - 1)/2$, then \Cref{conj:extended} holds with $e = 1$.
We are also able to prove \Cref{conj:extended} when $d \equiv r \pmod{p - 1}$ and all base-$p$ digits of $d$ are at least $r$.

\begin{lemma}
Let $n \geq 2$ and $d = (d_0,d_1,\dots,d_{n-1})_p$. If $d \equiv r \pmod{p - 1}$ with $2 \leq r \leq p - 1$ and $d_i \geq r$ for all $i$ in $\{0, 1, \ldots, n-1\}$, then there exists a positive integer $e \leq q - 1$ such that \[\wt(e \star d) - \wt(e) > \frac{n(p-1)}{2}.\]
\end{lemma}
\begin{proof}
Since $d \equiv \wt(d) \pmod{p-1}$, it follows that $\wt(d) = k(p-1) + r$ for a non-negative integer $k \leq n$. We can assume $k < n / 2$ as otherwise, $\wt(d) - 1 \geq n(p-1)/2 + 1$ so that $e = 1$ is sufficient. Multiplying $d$ by $e = \sum_{i=0}^{n-2} p^i$ yields
    \[ed \equiv [\wt(d) -d_1, \wt(d) - d_2, \dots, \wt(d) - d_{n-1},\wt(d) - d_{0}]_p \pmod {q-1}.\]
By substituting $\wt(d) = k(p - 1) + r$, we get 
    \begin{align*}
    ed &\equiv [kp+r-k -d_1, kp+r-k - d_2, \dots, kp+r-k - d_{n-1},kp+r-k - d_{0}]_p \\
    &\equiv [r -d_1, r - d_2, \dots, r - d_{n-1},r - d_{0}]_p \\
    &\equiv [p-1+r -d_1, p-1+r - d_2, \dots, p-1+r - d_{n-1},p-1+r - d_{0}]_p \pmod{q-1}.
    \end{align*}
Since $d_i\geq r$ for all $i$, then if $d \neq (r, r, \ldots, r)_p$, we have that $ed$ is congruent to
    \[(p-1+r -d_1, p-1+r - d_2, \dots, p-1+r - d_{n-1},p-1+r - d_{0})_p = e \star d.\]
If $d = (r,r,\dots,r)_p$, then we have $ed = q-1 = e \star d$.
Hence,
\begin{equation*}
\wt(e \star d) = n(p-1) + nr- \wt(d) = (n-k)(p-1) +r(n-1) > \frac{n(p-1)}{2} + r(n-1).
\end{equation*}
Since $\wt(e) = n-1$ and $r\geq 1$, the result follows. 
\end{proof}

This yields the following corollary on the non-planarity of monomials.
\begin{corollary}
Let $n\geq 2$ and $d = (d_0,d_1,\dots,d_{n-1})_p$. If $d_i > 1$ for all $i$ in $\{0, 1, \ldots, n-1\}$, then $X^d$ is not planar over $\F_{q}$.
\end{corollary}

\section*{Acknowledgements}

We thank the anonymous reviewers for their helpful comments and suggestions.

\end{document}